\documentclass{amsart}

\usepackage[all]{xy}
\usepackage[latin1]{inputenc}
\usepackage{hyperref}
\usepackage{amssymb}
\usepackage{tensor}
\newtheorem*{Maintheorem*}{Main Theorem}
\newtheorem*{theorem*}{Theorem}

\newtheorem{theorem}{Theorem}
\newtheorem{definition}{Definition}
\newtheorem{lemma}{Lemma}

\newtheorem{corollary}{Corollary}

\newcommand{\A}{\mathcal{A}}

\newcommand{\C}{\mathbb{C}}

\newcommand{\la}{\alpha}
\newcommand{\lb}{\overline{\beta}}

\newcommand{\tr}{\operatorname{tr}}

\newcommand{\ovl}{\overline}
\newcommand{\dbar}{\bar \partial}
\newcommand{\dl}{ \partial}

\renewcommand\>{\rangle}

\newcommand{\X}{\mathcal{X}}

\newcommand{\jbar}{\ovl{\jmath}}

\newcommand{\lbar}{\ovl{l}}

\newcommand{\Abar}{\ovl{A}}
\newcommand{\Bbar}{\ovl{B}}

\newcommand{\id}{\operatorname{id}}

\newcommand{\laplace}{\Box_\partial}    
\newcommand{\laplacedbar}{\Box_{\dbar}}

\begin{document}

\title{Twisted Hodge filtration: Curvature of the determinant}

\author{Philipp Naumann}
\address{Fachbereich Mathematik und Informatik,
Philipps-Universit\"at Marburg, Lahnberge, Hans-Meerwein-Straße, D-35032
Marburg,Germany}
\email{naumann@mathematik.uni-marburg.de}

\thanks{}



\date{}

\begin{abstract}
Given a holomorphic family $f:\mathcal{X} \to S$ of compact complex manifolds and a relative ample line bundle $L\to \X$, the higher direct images $R^{n-p}f_*\Omega^p_{\X/S}(L)$ carry a natural hermitian metric. An explicit formula for the curvature tensor of these direct images is given in \cite{Na16}. We prove that the determinant of the twisted Hodge filtration $F^p_L=\oplus_{i\geq p}R^{n-i}\Omega^i_{\X/S}(L)$ is (semi-) positive on the base $S$ if $L$ itself is (semi-) positive on $\X$.
\end{abstract}

\maketitle

\section{introduction}
For an ample line bundle $L$ on a compact complex manifolds $X$ of dimension $n$, the cohomology groups $H^{n-p}(X,\Omega^p_X(L))$ are critical with respect to Kodaira-Nakano vanishing. More generally, we consider the higher direct image sheaves $R^{n-p}f_*\Omega^p_{\X/S}(L)$ for a proper holomorphic submersion $f: \X \to S$ of complex manifolds and a line bundle $L \to \X$, which is positive along the fibers $X_s=f^{-1}(s)$. The understanding of this situation has applications to moduli problems. An explicit curvature formula for these higher direct images is given in \cite{Na16}. In general, the curvature of the intermediate higher direct images contains positive and negative contributions. Therefore, we consider the twisted Hodge filtration 
$$
F^p_L:=\oplus_{i\geq p}{R^{n-i}\Omega^i_{\X/S}(L)}
$$
This allows us to offset the negative terms by the positive ones. The result is that $\det(F^p_L)$ is (semi-) positive on the base $S$ if $L$ is (semi-) positive and relatively ample on $\X$. This idea comes from \cite{Sch16}.

\section{Differential geometric setup and proof of the result}
First we recall the setting from \cite{Na16}.
Let $f: \X \to S$ be a proper holomorphic submersion and $(L,h)$ a line bundle on $\X$. The curvature form of the hermitian line bundle is given by
$$
\omega_{\X}:=2\pi \cdot c_1(L,h)=-\sqrt{-1}\dl\dbar\log h.
$$ 
We consider the case where the hermitian bundle $(L,h)$ is relative ample, which means that
$$
\omega_{X_s}:=\omega|_{X_s}
$$
are K\"ahler forms on the $n$-dimensional fibers $X_s$. Then one has the notion of a horizontal lift $v_s$ of a tangent vector $\dl_s$ on the base $S$ and we get a representative of the Kodaira-Spencer class 
$$
A_s:=\dbar(v_s)|_{X_s}.
$$
Note that we have $L_{v_s}(\omega_{X_s})^n=0$ (see \cite{Be11}). Furthermore, one sets
$$
\varphi:=\<v_s,v_s\>_{\omega_{\X}},
$$
which is called the geodesic curvature. The coherent sheaf $R^{n-p}f_*\Omega^p_{\X/S}(L)$ is locally free on $S$ outside a proper subvariety. In the case $n=p$ and $L$ ample, the sheaf $f_*(K_{\X/S}\otimes L)$ is locally free by the Ohsawa-Takegoshi extension theorem (see \cite{Be09}). We assume the local freeness of
$$
R^{n-p}f_*\Omega^p_{\X/S}(L)
$$
in the general case, hence the statement of the Grothendieck-Grauert comparison theorem holds. Now Lemma 2 of \cite{Sch12} applies, which says that we can represent local sections of $R^{n-p}f_*\Omega^p_{\X/S}(L)$ by $\dbar$-closed $(p,n-p)$-forms on the total space, whose restrictions to the fibers are harmonic. Let $\{\psi^1,\ldots,\psi^r\}$ be a local frame of the direct image consisting of such sections around a fixed point $s \in S$. We denote by $\{(\dl/\dl s_i)|i=1,\ldots,\dim S\}$ a  basis of the complex tangent space $T_sS$ of $S$ over $\C$, where $s_i$ are local holomorphic coordinates on $S$. The natural inner product is given by
$$
(\psi^k,\psi^l) = \int_X{\psi^k_{A_p\Bbar_q}\psi^{\lbar}_{C_q\ovl{D}_p}g^{\ovl{D}_pA_p}g^{\Bbar_{q}C_q}h \;g\,dV},
$$
Let $A_{i\lb}^{\la}(z,s)\dl_{\la}dz^{\lb}=\dbar(v_i)_{X_s}$ be the $\dbar$-closed representative of the Kodaira-Spencer class of $\dl_i$ described above. Then the cup product together with contraction defines
\begin{eqnarray}
A_{i\lb}^{\la}\dl_{\la}dz^{\lb}\cup \; : \A^{0,n-p}(X_s,\Omega^p_{X_s}(L|_{X_s})) \to \A^{0,n-p+1}(X_s,\Omega^{p-1}_{X_s}(L|_{X_s})) \label{A_p}\\
A_{\jbar\la}^{\lb}\dl_{\lb}dz^{\la}\cup \; : \A^{0,n-p}(X_s,\Omega^p_{X_s}(L|_{X_s})) \to \A^{0,n-p-1}(X_s,\Omega^{p+1}_{X_s}(L_{X_s}))\label{A_p*} 
\end{eqnarray}  
where $p>0$ in (\ref{A_p}) and $p<n$ in (\ref{A_p*}). Note that this is a formal analogy to the derivative of the period map in the classical case (see \cite{Gr70}). We will apply the above cup products to harmonic $(p,n-p)$-forms. In general, the results are not harmonic. 

When applying the Laplace operator to $(p,q)$-forms with values in $L$ on the fibers $X_s$, we have 
\begin{equation}
\label{BKN}
\laplace - \laplacedbar = (n-p-q)\cdot \id
\end{equation}
due to the definition $\omega_{X_s}=\omega_{\X}|_{X_s}$ and the Bochner-Kodaira-Nakano identity. Thus, we write $\Box=\laplace=\laplacedbar$ in the case $q=n-p$. By considering an eigenfunction decomposition and using the identity (\ref{BKN}), we obtain that all eigenvalues of 
$\Box$ are $0$ or greater than $1$, hence the operator $(\Box-1)^{-1}$ exists. We use the notation 
$\psi^{\lbar}:=\ovl{\psi^l}$ for sections $\psi^l$ and write $g\,dV =\omega_{X_s}/n!$. The result from \cite{Na16} is 
\begin{theorem}
\label{Thm}
Let $f: \X \to S$ be a proper holomorphic submersion and $(L,h) \to \X$ a relative ample line bundle. With the objects described above, the curvature of $R^{n-p}f_*\Omega_{\X/S}^p(L)$ is given by 
\begin{eqnarray*}
R_{i\jbar}^{\lbar k}(s) = &&\int_{X_s}{\varphi_{i\jbar}\cdot(\psi^k\cdot\psi^{\lbar})\,g\,dV}\\
&+& \int_{X_s}{(\Box +1)^{-1}(A_i \cup \psi^k)\cdot(A_{\jbar}\cup\psi^{\lbar})\,g\,dV}\\
&+& \int_{X_s}{(\Box -1)^{-1}(A_i \cup \psi^{\lbar})\cdot(A_{\jbar}\cup\psi^{k})\,g\,dV}\\
\end{eqnarray*}
If $L\to \X$ is ample, the only contribution, which may be negative, originates from the harmonic parts in the third term
$$
-\int_{X_s}{H(A_I\cup\psi^{\lbar})\ovl{H(A_j\cup\psi^{\ovl{k}})}\,g\,dV}.
$$ 
\end{theorem} 
From now on, we assume that the base space $S$ is one dimensional. Let $\dl_s$ be a tangent vector and $A=A_s$ the corresponding 
$\dbar$-closed Dolbeault representative of the Kodaira-Spencer class. By considering an eigenvector decomposition, we have the following estimate for the curvature of $R^{n-p}\Omega^p_{\X/S}(L)$ for a semi-positive, relatively ample line bundle $L$ (see also \cite[Corollary 5]{Sch16}):
\begin{corollary}
\begin{equation}
\label{TwGri}
R(A,\Abar,\psi,\ovl{\psi})\geq ||H(A\cup \psi)||^2 - ||H(\Abar \cup \psi)||^2
\end{equation}
\end{corollary}
Now we formalize the situation:
The cup products (\ref{A_p}) and (\ref{A_p*}) give rise to cup products for cohomology classes ($\psi$ is assumed to be harmonic).
\begin{eqnarray}
A_p\;  : H^{n-p}(X_s,\Omega^p_{X_s}(L|_{X_s})) \to H^{n-p+1}(X_s,\Omega^{p-1}_{X_s}(L|_{X_s})) \quad ; \quad [\psi] \mapsto 
[H(A\cup \psi)]
\label{A}\\
\Abar_p \; : H^{n-p}(X_s,\Omega^p_{X_s}(L|_{X_s})) \to H^{n-p-1}(X_s,\Omega^{p+1}_{X_s}(L_{X_s})) \quad ; \quad [\psi] \mapsto 
[H(\Abar\cup \psi)]
\label{Abar} 
\end{eqnarray}  
We have (\cite[Lemma 5]{Sch16})
\begin{lemma}
$\Abar_{p}=(A_{p+1})^*$, i.e. the map $\Abar_{p}$ is the adjoint of the map $A_{p+1}$.
\end{lemma}
\begin{proof}
For any harmonic $\chi \in \A^{(p,n-p)}(X_s,L|_{X_s})$ we have $(H(\Abar \cup \psi),\chi) = (\Abar \cup \psi,\chi)=(\psi,A\cup \chi)=(\psi,H(A\cup \chi))$.
\end{proof}
In the spirit of the Hodge filtration we define for $0\leq p\leq n$
\begin{definition}
$$
F^p_L:=\bigoplus_{i\geq p}{R^{n-i}\Omega^i_{\X/S}(L)}
$$
\end{definition}
Then we have
\begin{theorem}
The line bundle $\det(F^p_L)$ is (semi-) positive on $S$, if $L$ is (semi-) positive and relatively ample on $\X$.
\end{theorem}
\begin{proof}
We choose a local holomorphic frame $\{\psi^{\la}\}\subset \A^{(p,n-p)}(X_s,L|_{X_s})$ of harmonic forms and compute
\begin{eqnarray*}
\sum_{\la}{||H(A\cup \psi^{\la})||^2} &=& \sum{(H(A\cup \psi^{\la}),H(A \cup \psi^{\la}))}\\
&=& \sum{(H(\Abar\cup A \cup \psi^{\la}),\psi^{\la})}\\
&=& \sum{((A_p^*\circ A_p)\psi^{\la},\psi^{\la})}\\
&=& \tr (A_p^*A_p)
\end{eqnarray*}
Analogously
\begin{eqnarray*}
\sum_{\la}{||H(\Abar\cup \psi^{\la})||^2} &=& \sum{(H(\Abar\cup \psi^{\la}),H(\Abar \cup \psi^{\la}))}\\
&=& \sum{(H(A \cup \Abar \cup \psi^{\la}),\psi^{\la})}\\
&=& \sum{((\Abar_p^*\circ \Abar_p)\psi^{\la},\psi^{\la})}\\
&=& \tr (\Abar_p^*\Abar_p)\\
&=& \tr(A_{p+1}A_{p+1}^{*}),
\end{eqnarray*}
because of Lemma 1.
Note that we have $\tr(A_p^*A_p)=\tr(A_pA_p^*)$. Now the result follows from Corollary 1, since the curvature of the determinant of a bundle is the trace of the curvature form of the bundle.
\end{proof}

\end{document}